\documentclass[reqno,12pt]{amsart}
\usepackage{graphicx} % Required for inserting images
\usepackage{amsthm}
\usepackage{amsmath}
\usepackage{amssymb}
\usepackage{mathtools}
\usepackage{amsfonts}
\usepackage{xcolor}
\usepackage{fullpage}
\usepackage[hidelinks]{hyperref}
\usepackage[msc-links,initials]{amsrefs}
\usepackage{enumitem}
\usepackage{todonotes}

\newtheorem{theorem}{Theorem}[section]
\newtheorem{corollary}[theorem]{Corollary}
\newtheorem{lemma}[theorem]{Lemma}
\newtheorem{conjecture}[theorem]{Conjecture}

\theoremstyle{definition}
\newtheorem{definition}[theorem]{Definition}

\DeclareMathOperator{\lcm}{lcm}

\title{On the densities of covering numbers and abundant numbers}

\author{Nathan McNew}
\email{nmcnew@towson.edu}

\address{Department of Mathematics, Towson University, Towson, MD 21286}
\author{Jai Setty}
\email{settyjai@gmail.com}
\address{Department of Mathematics, New York University, New York, NY 10012}

\renewcommand{\baselinestretch}{0.995}

\begin{document}

\begin{abstract}
We investigate the densities of the sets of abundant numbers and of covering numbers, integers $n$ for which there exists a distinct covering system where every modulus divides $n$. We establish that the set $\mathcal{C}$ of covering numbers possesses a natural density $d(\mathcal{C})$ and prove that $0.103230 < d(\mathcal{C}) < 0.103398.$ Our approach adapts methods developed by Behrend and Del\'eglise for bounding the density of abundant numbers, by introducing a function $c(n)$ that measures how close an integer $n$ is to being a covering number with the property that $c(n) \leq h(n) = \sigma(n)/n$.  However, computing $d(\mathcal{C})$ to three decimal digits requires some new ideas to simplify the computations. As a byproduct of our methods, we obtain significantly improved bounds for $d(\mathcal{A})$, the density of abundant numbers, namely $0.247619608 < d(\mathcal{A}) < 0.247619658$. We also show the count of primitive covering numbers up to $x$ is $O\left( x\exp\left(\left(-\tfrac{1}{2\sqrt{\log 2}} + \epsilon\right)\sqrt{\log x} \log \log x\right)\right)$,
which is substantially smaller than the corresponding bound for primitive abundant numbers. 
\end{abstract}
\maketitle

\section{Introduction}

A finite collection of arithmetic progressions is a \textit{covering system} if every integer is contained in at least one of its progressions.  More formally, a covering system is a set of ordered pairs $(a_i,m_i)$ of positive integers where every integer satisfies at least one of the congruences
\[
n \equiv a_i \!\!\!\!\pmod{m_i}.
\]
A covering system is \textit{distinct} if each modulus is used at most once.  For example, one can easily check that 
\begin{equation}
    \{0 \!\!\!\!\pmod{2}, \ 0 \!\!\!\!\pmod{3}, \ 1 \!\!\!\!\pmod{4}, \ 1 \!\!\!\!\pmod{6}, \ 11 \!\!\!\!\pmod{12} \} \label{eq:12ex}
\end{equation} 
form a distinct covering system. 

An integer $n$ is a \textit{covering number} if there exists a distinct covering system where every modulus is a divisor of $n$ that is greater than 1.  The example in \eqref{eq:12ex} shows that 12 is a covering number, and, in fact, 12 is the least positive covering number.  It follows from the definition that any multiple of a covering number is also a covering number, so we are especially interested in those covering numbers which are not multiples of smaller covering numbers. We call $n$ a \textit{primitive covering number} when $n$ is a covering number but is not divisible by any smaller covering number.  

This terminology is highly reminiscent of abundant numbers, which, as we will see, are closely related to the covering numbers.  A number $n$ is \textit{abundant} if $\sigma(n)>2n$ (where $\sigma$ denotes the sum of divisors of $n$).  Since $\sigma(n)$ is multiplicative, $\sigma(p^n)/p^n > \sigma(p^m)/p^m$ whenever $n>m$ for any prime $p$, and $\sigma(p)/p>1$, every multiple of an abundant number is also abundant. So an integer $n$ is called \textit{primitive abundant} if $\sigma(n)\geq 2n$ but no proper divisor of $n$ is abundant.\footnote{Note, this definition includes the perfect numbers, those $n$ with $\sigma(n)=2n$ among the primitive abundant numbers.  Sometimes these numbers are more properly referred to as \textit{primitive non-deficient numbers}; however, for historical reasons we will use the term primitive abundant numbers.}

\subsection{Notation}
We denote by $D(n)$ the set of divisors of an integer $n$, by $\tau(n)=|D(n)|$ their count, by $\omega(n)$ the number of distinct prime factors and by $\Omega(n)$ the total number counted with multiplicity.  We write $P^+(n)$ and $P^-(n)$ for the largest and smallest prime factors of $n$ respectively.  For convenience we define $P^+(1)=1$ and $P^-(1)=\infty$.  An integer is $y$-smooth if $P^+(n)\leq y$ and $y$-rough if $P^-(n)\geq y$.  Throughout $p$ and $q$ will denote primes,  $p_1,p_2,\ldots$ will be used to denote the specific prime numbers that divide an integer $n$, while $q_1, q_2,
\ldots$ will represent the increasing sequence of all prime numbers.

We let  $\sigma(n) = \sum_{d|n} d$ be the usual sum of divisors function and let $h(n) = \frac 1n \sigma(n)$ be the ``abundancy index'' of the integer $n$.  We will denote the natural density of a set $\mathcal{S}$ (when it exists) by $d(\mathcal{S})$, the sets of abundant and covering numbers will be denoted by $\mathcal{A}$ and $\mathcal{C}$, and the set of primitive members of each will be denoted $\mathcal{P}_{\mathcal A}$ and $\mathcal{P}_{\mathcal C}$ respectively.  The counting function of a set $\mathcal{S}$ will be denoted $\mathcal{S}(x)$.

\subsection{Background}
Both covering systems and covering numbers have a rich history with famous problems, many of which have seen significant progress in recent years, see for example \cite[(F13)]{guy2004} for some history, and \cite{Balister24} for a survey of some of the recent developments.

Covering numbers were first defined by Haight \cite{Haight} in 1979 in a paper answering a question of Erd\H{o}s, who asked (in the language of covering numbers) whether there was a threshold $c$ such that if $h(n) = \frac{\sigma(n)}{n}>c $ then $n$ is guaranteed to be a covering number.  Haight showed this is not the case by exhibiting an infinite sequence of numbers $n_i$, none of which are covering, with $\frac{\sigma(n_i)}{n_i} \to \infty$.  Filaseta, Ford, Konyagin, Pomerance and Yu \cite{FFKPY} gave a short proof of Haight's theorem and also answered several related conjectures of Erd\H{o}s and Graham. 

Sun \cite{sun07} initiated a systematic investigation of covering numbers, proving there are infinitely many primitive covering numbers, and giving the following sufficient conditions (restated slightly) for an integer $n$ to be a primitive covering number.
\begin{theorem}[Sun] \label{thm:sunPrimitive}
Suppose an even integer $n$ factors as $n=p_1^{\alpha_1}p_2^{\alpha_2}\cdots p_k^{\alpha_k}p_{k+1}$ with $k\geq 1$, $\alpha_i\geq 1$ for each $1\leq i \leq k$ and $2=p_1<p_2<\cdots<p_k<p_{k+1}$.  If these factors satisfy
\begin{enumerate}
    \item $p_{i+1} = \tau(p_1^{\alpha_1}\cdots p_i^{\alpha_i})+1$ \,for each $1\leq i< k$,
    \item $p_{k+1}\leq \tau(p_1^{\alpha_1}p_2^{\alpha_2}\cdots p_k^{\alpha_k})$, and 
    \item $p_{k+1}>(p_k-2)(p_k-3)$,
\end{enumerate}
then $n$ is a primitive covering number.
\end{theorem}
He also conjectured that all primitive covering numbers have a form slightly more general than the above sufficient conditions.  This conjecture was disproven by Jones and White \cite{Jones17}, who demonstrated an infinite family of primitive covering numbers which were not of the form conjectured by Sun.  Subsequently Jones, Harrington, and Phillips \cite{harrington17} found infinitely many more counterexamples and provided necessary and sufficient conditions for a number of the form $2^ap^bq$ to be a primitive covering number.

Many more general problems and results about covering systems can be framed in terms of covering numbers.  One of the most well-known open problems related to covering systems is the conjecture of Erd\H{o}s and Selfridge which is equivalent to the following.

\begin{conjecture}[Erd\H{o}s-Selfridge]
    There are no odd covering numbers.
\end{conjecture}

While the conjecture remains open, recent advancements \cites{Hough19,BBMST} show that any odd covering number would necessarily be divisible by 9 or 15, and furthermore that there are no odd squarefree covering numbers.

\subsubsection{Abundant Numbers}

Davenport \cite{davenport33} first proved that the set of abundant numbers possesses a natural density without obtaining explicit numerical bounds for it.  The following year, Erdős \cite{erdos34} gave an elementary proof of the existence of this density based on a fundamental observation: the density exists if and only if the sum $\sum_{a \in \mathcal{P}_{\mathcal{A}}} \frac{1}{a}$ of reciprocals of primitive abundant numbers converges, an idea which we will adapt for covering numbers.
The quest for explicit numerical bounds of this density began with the work of Behrend \cites{behrend32,behrend33}. In 1932, Behrend first showed that $d(\mathcal{A}) < 0.47$, and in fact proved the stronger result that $\mathcal{A}(n)/n < 0.47$ for all integers $n$. The following year, he refined his methods to obtain the first two-sided bounds $0.241 < d(\mathcal{A}) < 0.314$.
Behrend's method involves partitioning integers according to their largest smooth divisor and applying analytic estimates to bound the contribution from each part, which we describe further in Section \ref{sec:behrenddeleglise}. 
Subsequent improvements came from Salié \cite{salie55}, who improved the lower bound to $0.246$, and Wall et al.\ \cites{wall71,wall72}, who found $0.2441 < d(\mathcal{A}) < 0.2909$ (while their lower bound was not as strong as Salié's, they determined the first digit of this constant). The bounds were substantially refined by Del\'eglise \cite{deleglise}, who established  the bounds $0.2474 < d(\mathcal{A}) < 0.2480$, providing the first three digits.
Most recently, Kobayashi \cites{kobayashiThesis, kobayashi14} further improved Del\'eglise's approach to obtain four digits with the bounds \begin{equation}
    0.2476171 < d(\mathcal{A}) < 0.2476475. \label{eq:kobayashibds}
\end{equation} 

Klyve, Pidde, and Temple \cite{klyve19} investigated the density of abundant numbers experimentally under the assumption that the density of abundant numbers in long intervals behaves like independent, identically distributed variables with mean $d(\mathcal{A})$, though they also gave some reasons to question this assumption.  Assuming this hypothesis, their data suggest that, with high confidence, $d(\mathcal{A})$ lies in the interval (0.247619274, 0.247619713). We find below that $d(\mathcal{A})$ does in fact lie in (the upper part of) this range.

\section{Results}

\begin{theorem} \label{thm:covdensity}
    The set $\mathcal{C}$ of covering numbers has a natural density, and  \[0.103230 < d(\mathcal{C}) < 0.103398.\]
\end{theorem}

As mentioned above, the covering numbers share a close relationship with abundant numbers and our methods to compute this bound draw heavily from  the methods developed to bound the density of abundant numbers.  At the same time, one of our ideas leads to a substantially more efficient algorithm to compute the density of abundant numbers, which allows for the following improved bounds, which give us the first 7 digits of $d(\mathcal{A})=0.2476196\ldots$.

\begin{theorem} \label{thm:abundantbds} The natural density of the set of abundant numbers satisfies the bounds
    \[0.247619608 <d(\mathcal{A})< 0.247619658.\]
\end{theorem}

Much of this paper concerns the calculations that lead to the bounds for $d(\mathcal{C})$ and $d(\mathcal{A})$ listed above.\footnote{C++ code used to implement these methods is available at \url{https://github.com/agreatnate/AbunDens}}  We first need to establish that the set $\mathcal{C}$ has a density, which will follow from the following bound on the counting function of primitive covering numbers.

\begin{theorem} \label{thm:primCovNums}
    For $\epsilon >0$ and sufficiently large $x$, the count of the primitive covering numbers up to $x$ is bounded above by \begin{equation} \mathcal{P}_{\mathcal C}(x) \leq x\exp\left(\left(-\tfrac{1}{2\sqrt{\log 2}} +\epsilon\right)\sqrt{ \log x}\log \log x\right). \label{eq:primCovBd} \end{equation}
\end{theorem}

It is interesting to compare this to the known bounds for the count of primitive abundant numbers.  In 1935, Erd\H{o}s \cite{erdos35} proved that for sufficiently large $x$
\[ x  \exp\left(-c_1\sqrt{\log x \log \log x}\right) \leq  \mathcal{P}_{\mathcal A} (x) \leq x\exp\left(-c_2\sqrt{\log x \log \log x}\right) \]
with $c_1 = 8$ and $c_2 = 1/25$. In 1985, Iv\'ic \cite{ivic85} improved this, by showing that these constants can be replaced by $c_1 = \sqrt{6} + \epsilon$ and $c_2 = 1/\sqrt{12} - \epsilon$. Most recently, Avidon \cite{avidon} improved both constants further, showing one can take $c_1=\sqrt{2} + \epsilon$ and $c_2=1-\epsilon$.  As a corollary, we see that there are far more primitive abundant numbers than primitive covering numbers.

\begin{corollary} We have $\lim \limits_{x \to \infty} \frac{\mathcal{P}_{\mathcal A}(x)}{\mathcal{P}_{\mathcal C}(x)} = \infty$. In fact, for $\epsilon>0$ and sufficiently large $x$  
    \[\frac{\mathcal{P}_{\mathcal A}(x)}{\mathcal{P}_{\mathcal C}(x)} > \exp\left( \left(\tfrac{1}{2\sqrt{\log 2}} -\epsilon\right)\sqrt{ \log x}\log \log x\right).\]
\end{corollary}
We note that the fact that there are fewer primitive covering numbers than primitive abundant numbers (and similarly that the density of the covering numbers exists) does not follow just from the fact that the covering numbers form a subset of the abundant numbers.\footnote{An example of a primitive subset of the abundant numbers whose counting function grows faster than $\mathcal{P}_\mathcal{A}(x)$ is the set of numbers of the form $6p$ where $p$ is prime.  By modifying this construction we can find primitive subsets of the abundant numbers whose sets of multiples do not have an asymptotic density.}

Unlike the primitive abundant numbers, we have not been able to obtain a lower bound for $\mathcal{P}_C(x)$ of the same shape as the upper bound in \eqref{eq:primCovBd}.  It follows from Theorem \ref{thm:sunPrimitive} that, for every odd prime $p$, the integer $n=2^{p-1}p$ is a primitive covering number.\footnote{For comparison, a number of the form $2^kp$ is a primitive abundant number whenever $2^{k}< p <2^{k+1}$.  Thus, up to $x$ there are asymptotically $ \gg \frac{\sqrt{x}}{\log x}$ many primitive abundant numbers of this form up to $x$, while there are only $\gg \frac{\log x}{\log \log x}$ many primitive covering numbers of this form.} Thus, $\mathcal{P}_\mathcal{C}(x)\gg \frac{\log x}{\log \log x}$. Using Lemma \ref{lem:largestprime}, we obtain the following strengthening of Theorem \ref{thm:sunPrimitive}.

\begin{theorem} \label{thm:strongersun} If an even integer $n$ satisfies the hypotheses of Theorem \ref{thm:sunPrimitive} with condition (3) stated there replaced with
\begin{enumerate}
\setcounter{enumi}{2}
    \item for each $1\leq i \leq k$, $p_{k+1}>\tau\left(\frac{n}{p_ip_{k+1}}\right)$,
\end{enumerate}
then $n$ is a primitive covering number.
\end{theorem}
Theorem \ref{thm:strongersun} is stronger than Theorem \ref{thm:sunPrimitive}, both in that this theorem implies Sun's result, and that, in practice, many more integers can be shown to be primitive covering numbers using this theorem.  For example, tabulating integers up to $10^{50}$ satisfying these conditions, we find that there are 1433 integers that satisfy the 3 sufficient conditions of Theorem \ref{thm:sunPrimitive} to be primitive covering numbers and 2\,592\,765 integers up to $10^{50}$ that are guaranteed to be primitive covering numbers using Theorem \ref{thm:strongersun} instead. However, we haven't yet been able to use this to improve on the lower bound for $\mathcal{P}_\mathcal{C}(x)$.  More data is given in Table \ref{tab:primcovcounts} in the Appendix.

\section{Primitive Covering Numbers}

The upper bound we obtain for the number of primitive covering numbers is a consequence of the following observation, which also follows from Lemma 2.1 of \cite{sun07}. 
\begin{lemma} \label{lem:largestprime}
    If $n$ is a primitive covering number then \begin{equation} P^+(n) \leq \tau\left(\frac{n}{P^+(n)}\right). \label{eq:primcovineq} \end{equation}
\end{lemma}
\begin{proof}
    Suppose $n$ is a primitive covering number, and let $p=P^+(n)$, with $p^k$ being the largest power of $p$ dividing $n$.  Then $n/p$ is not a covering number, meaning that at least one residue modulo $n/p$ is left uncovered by any residue system modulo the divisors of $n/p$.  This uncovered residue lifts to $p$ distinct residues left uncovered by the same system modulo $n$, each having a distinct remainder modulo $p^k$.   In order for $n$ to be a covering number, it is necessary to cover each of these residues using a modulus that is a divisor of $n$ but not a divisor of $n/p$.  Each such divisor is necessarily divisible by $p^k$, meaning any congruence class modulo such a divisor can cover at most one of these $p$ uncovered residues.  A covering will thus require at least $p$ divisors of $n$ which are not divisors of $n/p$, of which there are \[\tau(n) - \tau(n/p) \leq \tau(n/p).\]
From this, we conclude that $p \leq \tau(n/p)$.
\end{proof}

The key idea of the bound of Theorem \ref{thm:primCovNums} is that numbers typically have relatively few divisors, while the largest prime factor of an integer is typically relatively large, and so integers that satisfy the inequality \eqref{eq:primcovineq} are rare.  To make this precise, we will need estimates for the counts of smooth numbers and for numbers with many divisors.  A well-known bound for smooth numbers is that $\psi(x,y)$, the count of integers up to $x$ whose largest prime divisor is at most $y$, satisfies \begin{equation}
    \psi(x,y) = x \exp\left((-1+o(1))\frac{\log x}{\log y} \log\left(\frac{\log x}{\log y} \right)\right) \label{eq:smoothnumapprox} \end{equation}
for a large range of $y$, as $x \to \infty$. De Bruijn \cite{deBruijn51} showed this approximation holds for $\exp\left((\log x)^{5/8+\epsilon}\right) \leq y \leq x$.  This range was subsequently extended to $( \log x)^{1+\epsilon} \leq y \
\leq x$ in \cite{CEP} (though we require only an upper bound, which was already given in a range sufficient for our purpose in \cite{deBruijn51}).
Norton \cite{Norton94} gives an estimate for $\Delta(x,y)$, the number of integers up to $x$ with $\tau(n) > y$, namely 
\begin{equation}
    \Delta(x,y) = x\exp\left((-1+o(1))\frac{\log y}{\log 2}\log \log y \right) \label{eq:largedivapprox}
\end{equation}
holds for $ (\log x)^{2\log 2+\epsilon}\leq y \leq x$.

\begin{proof}[Proof of Theorem \ref{thm:primCovNums}]  Set $y = \exp\left(\sqrt{\log 2 \log x}\right)$. If $n\leq x$ is a primitive covering number then by \eqref{eq:primcovineq} we cannot have both $P^+(n) > y$ and $\tau(n/P^+(n))\leq y$.  Hence, the integer $n$ is counted by at least one of $\psi(x,y)$ or $\Delta(x,y)$.  So, using \eqref{eq:smoothnumapprox} and \eqref{eq:largedivapprox}, we find that
\begin{align*}
    \mathcal{P}_{\mathcal C}(x) &\leq \psi(x,y) + \Delta(x,y) \\ &= x \exp\left((-1+o(1))\frac{\log x}{\log y} \log\left(\frac{\log x}{\log y} \right)\right) + x\exp\left((-1+o(1)\frac{\log y}{\log 2}\log \log y \right) \\
    &= x\exp\left((-\tfrac{1}{2} + o(1))\sqrt{\tfrac{1}{\log 2} \log x}\log \log x\right),
\end{align*}
which we can rewrite in the form of the bound \eqref{eq:primCovBd} in the statement of the theorem. \qedhere
    
\end{proof}

Importantly for our purposes, this bound suffices to show the following.  

\begin{corollary}
    The reciprocal sum $\sum\limits_{c \in \mathcal{P}_{\mathcal{C}}} \frac 1c$ of primitive covering numbers converges. 
\end{corollary}
Erd\H{o}s \cite{erdos34} first proved the analogous result for primitive abundant numbers, which he used to give an elementary argument that the natural density of the abundant numbers exists. (The existence of the density had already been shown by Davenport \cite{davenport33} a year prior.)  The convergence of this reciprocal sum of primitive covering numbers then gives, by the same argument as Erd\H{o}s, that the covering numbers have a density.

\begin{corollary}
    The natural density $d(\mathcal{C})\!$ exists.
\end{corollary}

Recently Lichtman \cite{lichtman18} has investigated the numerical value of the reciprocal sum of primitive abundant numbers, showing that it is bounded between 0.348 and 0.380.  Due to the difficulties in identifying primitive covering numbers (discussed further in Section \ref{sec:Lowbd} below), obtaining similar bounds for the reciprocal sum of primitive covering numbers is likely to be a much more difficult problem.  The sum of the reciprocals of the known covering numbers up to one million, listed in Table \ref{tab:primcovnums}, shows that the reciprocal sum of primitive covering numbers is at least 0.12380.

\section{Upper Bounds for \texorpdfstring{$c(n)$}{c(n)}}

We obtain upper bounds for the density of the set of covering numbers following the methods that have been developed to bound the density of the abundant numbers.  We introduce two functions which will allow us to compare the set of covering numbers to the set of abundant numbers. For a positive integer $n$ we first define $r(n)$ to be the maximum cardinality of the set of residues that can be covered by a set of congruences with distinct moduli greater than 1, each dividing $n$, and then set 
    \[c(n) \coloneqq 1 + \frac{r(n)}n.\]
    
Thus, by definition, an integer $n$ is a covering number if and only if $r(n)=n$ and $c(n) = 2$.  This function $c(n)$ is naturally related to the sum of divisors function as seen from the following lemma.

\begin{lemma} \label{lem:cboundh} For each $n$ we have $c(n) \leq \frac{\sigma(n)}{n} = h(n)$.
\end{lemma}

\begin{proof}
Each congruence class $k \pmod d$ with $d|n$ contains exactly $n/d$ residue classes modulo $n$. Therefore,
\begin{equation*}
 c(n) = 1 + \frac{r(n)}n \leq 1+ \frac{1}{n}\sum_{\substack{d|n \\d > 1}} \frac{n}{d} = \frac 1n \sum _{d|n} d = \frac{\sigma(n)}{n}. \qedhere
\end{equation*}
    
\end{proof}

Using a famous result of Newman \cite{newman} that any distinct covering system covers some congruence class multiple times, we can strengthen this, when $n$ is a covering integer, to $2=c(n)<h(n)$ and so we see that any covering number is abundant. 
\begin{corollary}
The covering numbers are contained in the abundant numbers, i.e.\ if $n$ is a covering number then $\sigma(n) > 2n$.
\end{corollary}

This was previously observed by Sun \cite{sun07}.  It follows from this corollary that the density of covering numbers is at most the density of abundant numbers, and so using the upper bound for the abundant numbers obtained by Kobayashi \cite{kobayashiThesis} we have that 
\begin{equation} \label{eq:firstabundantupperbound}
 d(\mathcal{C}) \leq d(\mathcal{A}) \leq 0.24765.\end{equation} 

We can improve the upper bound from Lemma \ref{lem:cboundh} in a variety of ways.  First we obtain the following improved bound for $c(n)$.  

\begin{theorem} \label{thm:cleqch}
For any coprime integers $m,n$ we have 
\[c(mn) \leq c(m)h(n) .\]
\end{theorem}

\begin{proof}
First, since the divisors of $m$ (exceeding 1) can cover at most $r(m)$ residues modulo $m$, those same divisors can cover at most $nr(m)$ residues modulo $nm$.  More trivially, each divisor $d$ (exceeding 1) of $n$ can cover at most $\frac{mn}{d}$ residues modulo $nm$. 

Now, fix a divisor $d$ of $n$, $d>1$, and consider all the moduli of the form $d\ell$ with $\ell|m$ and $\ell>1$. Since, as before, these divisors of $m$ can cover at most $r(m)$ residues modulo $m$, the corresponding congruences modulo $d\ell$ can cover at most $\frac{n}{d}r(m)$ residues modulo $nm$.   Thus, summing each of these contributions we have \begin{align*}
    c(mn) = 1 + \frac{r(mn)}{mn} &\leq  1+\frac{1}{mn}\Bigg(nr(m) + \sum_{\substack{d|n\\d>1}} \left(\frac{mn}{d}+\frac{nr(m)}{d}\right)\Bigg) \\
    &= 1+ \frac{r(m)}{m} + \sum_{\substack{d|n\\d>1}}\frac{1}{d} + \frac{r(m)}{m} \sum_{\substack{d|n\\d>1}}\frac{1}{d} \\
    &=\left(1+\frac{r(m)}{m}\right)\Bigg(1+\sum_{\substack{d|n\\d>1}}\frac{1}{d}\Bigg) \, = \ c(m)h(n). \qedhere
\end{align*}
\end{proof}

We conjecture that this can be strengthened.
\begin{conjecture}
The function $c(n)$ is sub-multiplicative, that is $c(nm) \leq c(n)c(m)$ whenever $\gcd(n,m)=1$. 
\end{conjecture}
This conjecture has been verified numerically for all $m,n$ with $mn<10\,000$.

\subsection{Almost-covering numbers}
As we will see below, our study of covering numbers (especially primitive covering numbers) will lead us to consider a closely related class of numbers \cite[A160560]{oeis} which, while not covering numbers, are as close as possible to being covering numbers.     

\begin{definition}
    We call an integer $n$ an \textit{almost-covering number} if $r(n)=n-1$.  In particular, $n$ is an almost-covering number if it is not a covering number, but it is possible to cover all but one of the residue classes modulo $n$ with distinct moduli corresponding to the divisors of $n$ greater than 1.  A collection of arithmetic progressions with distinct moduli that covers all but one residue modulo $n$ is an \textit{almost-covering}.
\end{definition}

Note that the integer 1 is, by our definition (somewhat vacuously), an almost-covering number.  Many of Sun's \cite{sun07} arguments implicitly use the idea of almost-covering numbers without defining them. A strengthening of his arguments gives the following.

\begin{theorem}\label{thm:sunalmostcover}
    Suppose $n>1$ factors as \begin{equation}n=p_1^{\alpha_1} p_2^{\alpha_2}\cdots p_k^{\alpha_k} \label{eq:factorn} \end{equation} with $2=p_1<p_2<\cdots<p_k$.  If, for each $1<i\leq k$, we have 
    \begin{equation} p_i \ =\ (\alpha_1+1)(\alpha_2+1)\cdots(\alpha_{i-1}+1) +1 \ = \ \tau\left(p_1^{\alpha_1} p_2^{\alpha_2}\cdots p_{i-1}^{\alpha_{i-1}}\right) +1 \label{eq:sunalmostcovercondition} \end{equation} 
    then $n$ is an almost-covering number.
\end{theorem}

Henceforth we refer to almost-covering numbers in the form of this theorem as \textit{Sun-almost-covering} numbers.  We will also frequently need to reference the largest (greedily constructed using small prime divisors) Sun-almost-covering divisor of an integer $n$, and we will denote this divisor $\ell = \ell(n)$.  We define this divisor formally below.\footnote{Note that in certain cases where $n$ is a covering number $\ell(n)$ may not be the largest divisor of $n$ that is Sun-almost-covering.  For example, taking $n=2 \times 3^4 \times 7 \times 11$, we find, by applying Definition \ref{def:sacd}, that $\ell(n) = 2 \times 3^4$, even though $d=2 \times 3^4\times 11$ is a divisor of $n$ that is Sun-almost-covering.} 

\begin{definition} \label{def:sacd}
    Write the factorization of $n$ as in \eqref{eq:factorn}. Set $\ell(n)=1$ if $n$ is odd and, if $n$ is even, take $j\leq k$ to be the largest index such that $p_i$ satisfies \eqref{eq:sunalmostcovercondition} for every $i\leq j$. We then define the \textit{Sun-almost-covering-divisor} of $n$ to be \[\ell(n) \coloneqq \begin{cases}
    1 & \text{$n $ is odd} \\
    2^{\alpha_1}p_2^{\alpha_2}\cdots p_j^{\alpha_j} & \text{$n $ is even.}
\end{cases} \]  
\end{definition}

Almost-covering numbers play a fundamental role in the distribution of primitive covering numbers.  For example, a short calculation shows the following. 
\begin{corollary} \label{cor:covfromalmostcov}
    If $n$ is an almost-covering number and $P^+(n)<p\leq \tau(n)$ then $pn$ is a covering number.
\end{corollary}
Note that $pn$ in the corollary above may not be a primitive covering number.  For example, $n=2^9\times 11$ is a Sun-almost-covering number, so, since $13<\tau(n)=20$, we find that $13n$ is covering, but it is not a primitive covering number since $2^8\times 11\times 13$ is also a covering number, as shown in \cite{Jones17}.  Applying Corollary \ref{cor:covfromalmostcov} to the special case of the almost-covering numbers obtained from Theorem \ref{thm:sunalmostcover}, we recover Theorem 1.1 of \cite{sun07}.  There are almost-covering numbers which are not Sun-almost-covering. The smallest we are aware of is $n=2^{10} \times 13 \times 17$.  Using this and Corollary \ref{cor:covfromalmostcov} we see, for example, that $2^{10} \times 13 \times 17 \times 41$ is a covering number.  We can even use a modified version of the proof of Theorem \ref{thm:strongersun} to show that it is a primitive covering number, which gives another counterexample to Sun's conjecture.

We can also now prove Theorem \ref{thm:strongersun}.

\begin{proof}[Proof of Theorem \ref{thm:strongersun}]
Suppose $n$ satisfies the first two hypotheses of Theorem \ref{thm:sunPrimitive}.  Then $n=\ell p_{k+1}$ where $\ell = p_1^{\alpha_1}p_2^{\alpha_2}\cdots p_k^{\alpha_k}$ is a Sun-almost-covering number and so $n$ is a covering number by Corollary \ref{cor:covfromalmostcov}.  

Thus, it remains only to show that $n$ is a primitive covering number, for which we use the new condition (3) of Theorem \ref{thm:strongersun}. Suppose for a contradiction that $n$ were not a primitive covering number.  Then $n$ has a primitive covering divisor $d|n$.  It must be the case that $p_{k+1}|d$ since otherwise $d|\ell$, and it isn't possible for a divisor of an almost-covering number to be covering.  Thus, since $d$ is a proper divisor of $n$, there must be some $p_i$, with $i\leq k$, such that $p_i$ divides $d$ fewer times than $n$.  But then, using the new Condition (3),  \[\tau\left(\frac{d}{p_{k+1}}\right) \leq \tau\left(\frac{n}{p_ip_{k+1}}\right)<p_{k+1}\]
but, using Lemma \ref{lem:largestprime},  this means $d$ is not a primitive covering number, a contradiction.
\end{proof}

Before we give a proof of Theorem \ref{thm:sunalmostcover}, we need the following Lemma which will be used several times in the remainder of the paper.

\begin{lemma} \label{lem:wlog-almost-cover}
If an integer $n$ factors as $n=\ell m$ with $\ell$, $m$ coprime and $\ell$ is an almost-covering number, then there exists a set of congruences using distinct moduli dividing $n$ which covers a maximal number, $r(n)$, of the residues modulo $n$ and the subset of congruences whose moduli divide $\ell$ form an almost-covering of the residues modulo $\ell$.  
\end{lemma}
\begin{proof}
    Fix a set $C$ of congruences with distinct moduli dividing $n$ which cover $r(n)$ residues modulo $n$, and consider those congruences in $C$ with moduli dividing $\ell$.  If they form an almost-covering of the residues modulo $\ell$ we are done, so suppose they do not.  Then there are at least two residues modulo $\ell$ left uncovered by these congruences.  Pick one of these uncovered residues, say $a \pmod{\ell}$.  Now, create a new system of congruences $C'$, consisting of an almost-covering of the residues modulo $\ell$ which covers every residue except $a \pmod{\ell}$, and all of the same congruences as in $C$ for each modulus that does not divide $\ell$.  
    
We now check that every residue modulo $n$ covered by $C$ is also covered by $C'$. Fix a residue modulo $n$ covered by $C$. If it was covered by a congruence with modulus dividing $\ell$ in $C$, then it is still covered by one of the congruences with moduli dividing $\ell$ in $C'$ (though possibly by a different congruence).  On the other hand, if it was covered by a congruence with moduli not dividing $\ell$ in $C$, it is still covered by the same congruence class in $C'$.   Thus, $C'$ is a set of congruences satisfying the conclusion of the lemma.  
\end{proof}

\begin{proof}[Proof of Theorem \ref{thm:sunalmostcover}]
    We prove that $n$ with factorization \eqref{eq:factorn}  satisfying the hypothesis of the theorem is an almost-covering number by induction on $k$. If $k=1$, then $n=2^{\alpha_1}$ and a simple computation shows that $n$ is an almost-covering number.  Now suppose that it holds for $k-1$, and consider $n=\ell p_k^{\alpha_k}$ where $\ell=p_1^{\alpha_1} p_2^{\alpha_2}\cdots p_{k-1}^{\alpha_{k-1}}$ is an almost-covering number by the induction hypothesis, and $p_k = \tau(\ell)+1$ by the hypotheses of the theorem.    

    Now, applying Lemma \ref{lem:wlog-almost-cover}, since $\ell$ is an almost-covering number coprime to $p_k$ we can assume, without loss of generality, that a set of congruences covering a maximal number of residues modulo $n$ forms an almost-covering of the residues modulo $\ell$ when restricted to the congruences with moduli coprime to $p_k$.  Thus, these congruences coprime to $p_k$ cover all but one arithmetic progression modulo $\ell$ and, by shifting all of the congruence classes as necessary, we can assume this uncovered progression is $0 \pmod{\ell}$. It remains to cover as much of this arithmetic progression as possible using moduli divisible by $p_k$.  The number of residues modulo $n$ in this progression covered by any given divisor $d$ of $n$ is $\frac{n \gcd(\ell,d)}{\ell d} = \frac{p_k^{\alpha_k} \gcd(d,\ell)}{d} = p_k^{\alpha_k-i}$, where $i$ is the number of times that $p_k$ divides $d$.

    For each $1\leq i \leq \alpha_k$ there are exactly $\tau(\ell)=p_k-1$ divisors of $n$ divisible by $p_k^i$ but not by $p_k^{i+1}$ and so we find that 
    \begin{align*}
        r(n) &\leq \left(\frac{\ell-1}{\ell}\right)n + (p_k-1)\sum_{i=1}^{\alpha_k} p_k^{\alpha_k-i}  \\&= n-\frac{n}{\ell} + (p_k-1)\sum_{i=0}^{\alpha_k-1} p_k^i = n-p_k^{\alpha_k} + (p_k-1)\frac{p_k^{\alpha_k}-1}{p_k-1}=n-1.
    \end{align*}  On the other hand, by choosing, for each of the $\tau(\ell)=p_k-1$ divisors $d_1,d_2,\ldots d_{\tau(\ell)}$ of $n$ exactly divisible by $p_k^i$, the congruence class $\ell \cdot p_k^{i-1} \cdot j \pmod{d_j}$ for $j \in \{1,2,\ldots \tau(\ell)\}$, we can see, between these congruences and the congruences modulo divisors of $\ell$, that every residue modulo $n$ is covered except the residue $0 \pmod{n}$.  Thus we can conclude that $n$ is an almost-covering number. \qedhere

\end{proof}

\subsection{Complementary Bell numbers and generalizations.}
If the moduli $m_1, m_2, \ldots m_k$ are coprime, it follows from the Chinese remainder theorem that the density of those integers contained in any residue system with these moduli is \[1- \prod_{1\leq i \leq k} \left( 1- \frac 1{m_i}\right) = \sum_{1\leq i \leq k} \frac{1}{m_i} - \sum_{1\leq i < j \leq k} \frac{1}{m_im_j} + \cdots + \frac{(-1)^{k+1}}{m_1m_2\cdots m_k}.\]

The following generalization follows likewise from the Chinese remainder theorem by a similar inclusion-exclusion argument.
\begin{lemma}  \label{lem:gencrdbd} Suppose that $M=\{m_1, m_2, \ldots m_k\}$ is any set (or multiset) of moduli. Then the proportion of integers covered by any residue system whose moduli are the elements of $M$ is at most  
\[\sum_{m \in M} \frac 1m \ - \sum_{\substack{m,m' \in M\\ \gcd(m,m') = 1}} \frac{1}{mm'} \ + \ \sum_{\substack{m,m'\!,m'' \in M \\ \text{ pairwise coprime}}} \frac{1}{mm'm''} \ - \ \cdots \ = \sum_{\substack{S \subseteq M, \ S \neq \emptyset \\ S \text{ pairwise coprime}}} \frac {(-1)^{|S|+1}}{\lcm  S}. \]

\end{lemma}

We now apply Lemma \ref{lem:gencrdbd} to the case where the moduli are the set $D_{>1}(n)\subset D(n)$ of divisors, exceeding 1, of some integer $n$.  This allows us to bound $c(n)$ as
\begin{align}
   c(n) = 1 + \frac{r(n)}{n} &\leq 1 + \sum_{\substack{S \subseteq D_{>1}(n), \ S \neq \emptyset, \\ S \text{ pairwise coprime}}} \frac {(-1)^{|S|+1}}{\lcm S} 
   \ = \ 1 + \sum_{d|n, d>1} \frac{1}{d} \!\! \sum_{\substack{S \subseteq D_{>1}(n) \\ \lcm S = d \\ S \text{ pairwise coprime}}} \!\!(-1)^{|S|+1} \notag \\
      &= 1 + \sum_{d|n, d>1} \frac{1}{d} \sum_{1\leq k \leq \omega(d)} (-1)^{k+1}S_2(\omega(d),k) \label{eq:stirlingbd}
\end{align} 
where $S_2(i,k)$ denotes the Stirling number of the second kind, which counts the number of ways to partition $i$ distinct objects (in our case the $i=\omega(d)$ distinct prime powers dividing $d$) into exactly $k$ nonempty subsets (the $k$ elements of $S$, each the products of the prime powers from each subset).  So, for any divisor $d|n$ we have \[S_2(\omega(d),k) =\sum_{\substack{S \subseteq D_{>1}(n) \\ \lcm S \ = \ d,\ |S|=k \\ S \text{ pairwise coprime}}} 1.\]   

Note that this innermost sum of \eqref{eq:stirlingbd} depends only on $\omega(d)$. Its values are the (negative of the) sequence of complementary Bell numbers (or Rao Uppuluri-Carpenter numbers).  In what follows, it will be useful to define a generalization of these numbers as follows.  We set
$$B(r,n) \coloneqq -\sum_{k=1}^n (-r)^k S_2(n, k).$$  So here, $B(1,n)$ is the negative of the sequence of complementary Bell numbers.

\begin{theorem} \label{thm:generalizedstirlingbd}
    If an integer $n$ factors as $n=\ell b$ where $\ell$ and $b$ are coprime and $\ell$ is an almost-covering number, then \begin{equation}
        c(n) \leq 1+\frac{\ell-1}{\ell} + \frac{1}{\ell} \sum_{d\mid b, d>1} \frac{B(\tau(\ell),\omega(d))}{d}. \label{eq:cstirlingupbd}
    \end{equation}
\end{theorem}

\begin{proof}
    Let $C$ be a set of congruences realizing $c(n)$.  By Lemma \ref{lem:wlog-almost-cover} we can assume, without loss of generality, that the congruence classes in $C$ form an almost-covering of the residues modulo $\ell$, when restricted to the divisors of $\ell$.  Furthermore, by adding a constant to each congruence, we can assume that those congruences in $C$ with moduli dividing $\ell$ cover every residue except $0 \pmod{\ell}$, so it remains to consider the arithmetic progression $0 \pmod{\ell}$.  So, for every congruence $r \pmod{m}$ in $C$ with modulus $m$ that is not a divisor of $\ell$, we can assume that $r \equiv 0 \pmod{\!\!\gcd(m,\ell)}$.  
    
    Thus, the residues modulo $n$ in the progression $0 \pmod{\ell}$ that are covered by these congruences are the same as those covered by the set $C'$ in which we replace each modulus $m\nmid \ell$ in $C$, with a modulus $m'= m/\!\gcd(m,\ell)$, and leave the residue class unchanged.  Note that the moduli in $C'$ are no longer distinct---each modulus occurs $\tau(\ell)$ times.
      
    Furthermore, since the moduli of $C'$ are now coprime to $\ell$, they are now divisors of $b$ (greater than 1 since $m\nmid \ell$), and the number of residues modulo $n$ in the progression $0 \pmod{\ell}$ covered by $C'$ are the same as the number of residues modulo $b$ covered by $C'$.  We thus find that 
    \begin{equation} c(n) = 1 + \frac{r(n)}{n} = 1 + \frac{\ell -1}{\ell} +\frac{1}{\ell}\left( \frac{r_{\tau(\ell)}(b)}{b}\right) \label{eq:firstcbound}
    \end{equation}
where $r_k(b)$ denotes the maximum cardinality of the residues modulo $b$ that can be covered by a set of congruences with moduli greater than 1 dividing $b$ in which each modulus appears at most $k$ times.  (Note that $r(b) = r_1(b)$.)  We conclude by bounding the expression $ \frac{r_{\tau(\ell)}(b)}{b}$ with a calculation analogous to that of \eqref{eq:stirlingbd}.  Applying Lemma \ref{lem:gencrdbd} to the multiset $M'$ of moduli in $C'$, we find that 
\begin{align*}
\frac{r_{\tau(\ell)}(b)}{b} \ &\leq \  \sum_{\substack{S \subseteq M', \ S \neq \emptyset \\ S \text{ pairwise coprime}}}   \frac{(-1)^{|S|+1}}{\lcm{S}}.
\end{align*}
Note that while $M'$ is a multiset, the pairwise-coprime condition on the summands $S$ ensures that each $S$ is in fact a set. Furthermore, since each divisor of $b$ appears with multiplicity $\tau(\ell)$ in $M'$, any fixed subset $S \subseteq D(b)$ appears $\tau(\ell)^{|S|}$ times in the sum above, and so we rewrite the expression above as 
\begin{align*}
\frac{r_{\tau(\ell)}(b)}{b} \ &\leq \  \sum_{\substack{S \subseteq D(b), \ S \neq \emptyset \\ S \text{ pairwise coprime}}}   \tau(\ell)^{|S|}\frac{(-1)^{|S|+1}}{\lcm{S}} = \ - \!\sum_{\substack{S \subseteq D(b), \ S \neq \emptyset \\ S \text{ pairwise coprime}\\ \ }} \frac{(-\tau(\ell))^{|S|}}{\lcm{S}}\\
&=-\sum_{d|b, d>1}\frac{1}{d} \sum_{\substack{S \subseteq D(b) \\ \lcm S \ = \ d \\ S \text{ pairwise coprime}}} (-\tau(\ell))^{|S|} = -\sum_{d|b, d>1} \frac{1}{d} \sum_{1\leq k \leq \omega(d)} (-\tau(\ell))^{k}\sum_{\substack{S \subseteq D(b) \\ \lcm S \ = \ d,\ |S|=k \\ S \text{ pairwise coprime}}} 1\notag \\
      &= -\sum_{d|b, d>1} \frac{1}{d} \sum_{1\leq k \leq \omega(d)} (-\tau(\ell))^{k}S_2(\omega(d),k) 
      \ = \ \sum_{d|b, d>1} \frac{B(\tau(\ell),\omega(d))}{d}.
\end{align*}
     Inserting this bound into \eqref{eq:firstcbound} completes the proof. 
\end{proof}

\section{Upper bounds for the density of covering numbers}

We are now able to improve on the bound \eqref{eq:firstabundantupperbound} by combining the bounds of Lemma \ref{lem:cboundh} and Theorem \ref{thm:generalizedstirlingbd} with techniques developed to bound the density of of abundant numbers from above.    We first recall briefly the approach introduced by Behrend \cites{behrend32,behrend33} and refined by Del\'eglise to obtain the bounds $0.2474 <d(\mathcal{A}) < 0.2480.$  We only sketch the argument here; see \cite{deleglise} or \cite{kobayashiThesis} for details of the calculations.

\subsection{The approach of Behrend-Del\'eglise} \label{sec:behrenddeleglise}

Let $y,z >0$ be fixed parameters and let \[A_y(x)\coloneqq d(\{n : P^-(n) \geq y, h(n)>x\})\] denote the density of the set of $y$-rough integers whose abundancy index exceeds $x$.  (Note that $d(\mathcal{A}) = A_2(2)$.)  Behrend's approach starts with the observation that, since any abundant number $a$ can be partitioned as $a=nm$ with $P^+(n)<y$ and $P^-(m)\geq y$, it must be the case that $2<h(a)=h(n)h(m)$.  Thus $m$, the $y$-rough component of $a$, must satisfy the inequality $h(m)>\frac{2}{h(n)}$, and thus be counted in the density $A_y\left(\frac{2}{h(n)}\right)$.  We can thus decompose the density $d(\mathcal{A})$, for any value of $y$, as
\begin{align}
\!d(\mathcal{A}) = \!\sum_{\substack{n \\  P^+(n) < y}} \frac{1}{n} A_y\left(\frac{2}{h(n)}\right) &= \!\sum_{\substack{n \leq z\\ \ P^+(n) < y}} \frac{1}{n} A_y\left(\frac{2}{h(n)}\right) + \sum_{\substack{n > z\\ \ P^+(n) < y}} \frac{1}{n} A_y\left(\frac{2}{h(n)}\right) \notag \\
&\leq\! \sum_{\substack{n \leq z\\ \ P^+(n) < y}} \frac{1}{n} A_y\left(\frac{2}{h(n)}\right) + \left(\prod_{p< y}\left(1-\frac{1}{p}\right)\right)\sum_{\substack{n > z\\ \ P^+(n) < y}} \frac{1}{n} \notag \\
&= \!\sum_{\substack{n \leq z\\ \ P^+(n) < y}} \frac{1}{n} A_y\!\left(\frac{2}{h(n)}\right) + \raisebox{-0.6ex}{$\displaystyle\left(\raisebox{0.5ex}{$\displaystyle 1-\prod_{p< y}\left(1-\frac{1}{p}\right)\!\!\!\sum_{\substack{n \leq z\\ \ P^+(n) < y}} \frac{1}{n} $}\right)$}. \label{eq:Behrendbd}
\end{align}   
In the second line above, the value of $A_y(x)$ is bounded from above by the density of all $y$-rough numbers, and in the third line, the second term, which denotes the density of those integers whose largest $y$-smooth divisor exceeds $z$, was rewritten as one minus the density of the complement of that set.  The terms $A_y(x)$ appearing in the first sum are then bounded using moments of the function $h_y(n) \coloneqq \sum\limits_{\substack{d|n \\ P^-(d)\geq y}} \frac{1}{d} = \prod\limits_{\substack{p^\alpha || n\\ p\geq y}}\left( 1+\frac{1}{p} + \cdots + \frac{1}{p^\alpha}\right)$.  A calculation (using that $h_y(n)\geq 1$ for all integers $n$) shows, for any $r>0$, that 
\begin{equation} A_y(x) \leq \frac{\mu_{y,r} -1}{x^r-1} \prod_{p< y}\left(1-\frac 1p\right)\label{eq:Aymomentbd}\end{equation}
where \begin{equation}
    \mu_{y,r} \coloneqq \lim\limits_{x \to \infty} \frac 1x \sum\limits_{n\leq x} h_y(n)^r = \prod\limits_{p\geq y}\left(1\, {+}\sum\limits_{i=1}^\infty\frac{\left(1+\frac 1p{+} \cdots {+} \frac 1{p^{i}}\right)^r\!-\left(1+\frac 1p{+} \cdots {+}\frac 1{p^{i-1}}\right)^r}{p^i}\right) \label{eq:moment}
\end{equation} is the $r$-th moment of $h_y(n)$.

Del\'eglise then obtains his numerical upper bound by taking $y=500$ and $z=10^{14}$ in \eqref{eq:Behrendbd} and \eqref{eq:Aymomentbd}, using explicit upper bounds for the terms in the product expression \eqref{eq:moment}, and finally taking for each application of \eqref{eq:Aymomentbd} the minimum value obtained as $r$ ranges over the powers of 2 between 1 and 4096.

\subsection{Initial bounds for the density of covering numbers}

By using the observation of Theorem \ref{thm:cleqch} in place of the multiplicativity of the function $h(x)$, we can immediately adapt Behrend's approach to obtain upper bounds for the density of the covering numbers.  Following the argument used to derive  \eqref{eq:Behrendbd} we see that 
\begin{align}
d(\mathcal{C}) \leq\! \sum_{\substack{n \\  P^+(n) < y}} \frac{1}{n} A_y\left(\frac{2}{c(n)}\right) 
&=\! \sum_{\substack{n \leq z\\ \ P^+(n) < y}} \frac{1}{n} A_y\left(\frac{2}{c(n)}\right) + \raisebox{-0.6ex}{$\displaystyle\left(\raisebox{0.5ex}{$\displaystyle 1-\prod_{p< y}\left(1-\frac{1}{p}\right)\!\!\!\sum_{\substack{n \leq z\\ \ P^+(n) < y}} \frac{1}{n} $}\right)$}. \label{eq:firstcovdenbd}
\end{align}
When trying to use this expression to obtain numerical upper bounds following Del\'eglise, one quickly finds that this bound is far less useful in practice than \eqref{eq:Behrendbd} is for abundant numbers, due to the relative difficulty of computing $c(n)$ vs $h(n)$.  While $h(n)$ can be computed very quickly, even for very large numbers so long as their factorization is known, exact computation of $c(n)$ quickly becomes intractable even for relatively small values of $n$ (see Section \ref{sec:Lowbd} for further discussion of the computation of $c(n)$.)

Since we seek here upper bounds for $d(\mathcal{C})$, we can replace $c(n)$ with an an easier-to-compute upper bound $c'(n)$ based on Theorem \ref{thm:generalizedstirlingbd}.  To make effective use of this theorem, we would like to take a large almost-covering divisor of $n$, however determining rigorously the largest almost-covering divisor of $n$ is difficult in general.  Instead, we take, for a fixed integer $n$, the largest almost-prime divisor in the form of Theorem \ref{thm:sunalmostcover} and define the following.

\begin{definition} \label{def:cprime}
    To define the function $c'(n)$ we first set $\ell=\ell(n)$ to be the Sun-almost-covering divisor of $n$ as defined in Definition \ref{def:sacd}.  We then  write $b=\frac n\ell = p_{j+1}^{\alpha_{j+1}}\cdots p_k^{\alpha_k}$.  
 
 If $P^-(b)=p_{j+1}\leq \tau(\ell)$ then $n$ is a covering number by Corollary \ref{cor:covfromalmostcov} and we define $c'(n)=2$, otherwise we define \[c'(n) = 1+\frac{\ell-1}{\ell} + \frac{1}{\ell} \sum_{d\mid b, d>1} \frac{B(\tau(\ell),\omega(d))}{d}, \]
    which is the expression on the right-hand side of \eqref{eq:cstirlingupbd}.
\end{definition}
Note that by construction (and Theorem \ref{thm:generalizedstirlingbd}) we have $c(n)\leq c'(n)$.\footnote{Unlike $c(n)$, for which we have $c(n)\leq 2$ by definition, it is possible  that $c'(n)>2$, even for $n$ which are not covering numbers. However this won't have an effect on our arguments.}  We can thus replace each occurrence of $c(n)$ in \eqref{eq:firstcovdenbd} with $c'(n)$ and obtain a new upper bound for $d(\mathcal{C})$ which is substantially more useful in practice.  
\begin{align}
d(\mathcal{C}) \leq\! \sum_{\substack{n \\  P^+(n) < y}} \!\frac{1}{n} A_y\left(\frac{2}{c'(n)}\right) 
&=\!\! \sum_{\substack{n \leq z\\ \ P^+(n) < y}} \!\frac{1}{n} A_y\left(\frac{2}{c'(n)}\right) + \raisebox{-0.6ex}{$\displaystyle\left(\raisebox{0.5ex}{$\displaystyle 1-\prod_{p< y}\left(1-\frac{1}{p}\right)\!\!\!\sum_{\substack{n \leq z\\ \ P^+(n) < y}} \frac{1}{n} $}\right)$}.\label{eq:secondcovdenbd}
\end{align} Nevertheless, computation of $c'(n)$ is still much slower than $h(n)$, and so we now make an additional improvement which will correspond to an analogous improvement to the application of \eqref{eq:Behrendbd} when computing $d(\mathcal{A})$ as well.  We present the improved computation of $d(\mathcal{A})$ first, for simplicity, before noting a few modifications necessary for the computation of $d(\mathcal{C})$.

\subsection{A more flexible version of Del\'eglise's bound} \label{sec:flexBehrend} 

The approach taken by Del\'eglise is to sum, as in \eqref{eq:Behrendbd}, over all $n$ in a ``large box,''  namely  all integers $n$ satisfying both $n \leq z$ and $P^+(n)< y$, with increasingly more precise numerical bounds obtained by taking larger values of $z$ and $y$ (at the cost of longer calculations).  When implementing this approach, to bound $d(\mathcal{C})$ using \eqref{eq:secondcovdenbd} with even moderate-sized values of $z$ and $y$, one quickly observes that the algorithm spends a substantial amount of time dealing with values of $n$ which are relatively inconsequential to the sum (large values of $n$ with $c'(n)$ relatively small) as well as integers $n$ which cannot correspond to any covering numbers (such as $n$ coprime to 6, which we know cannot be covering numbers after the work of Hough and Nielsen \cite{Hough19}).  The more substantial contribution to the sum comes from values of $n$ with $c'(n)$ close to 2, which typically correspond to $n$ divisible by high powers of 2.  This observation leads to the following, more general version of \eqref{eq:Behrendbd}.  We first explain how the idea works for abundant numbers, and then discuss how to treat covering numbers by a similar approach.

\begin{definition}
    For an ordered pair of integers $(a,q)$ satisfying $P^+(a) \leq q$, we denote by \[M_{a,q} \coloneqq \{av: P^-(v) \geq q\}\] the set of $q$-rough multiples of $a$.  Now, suppose that $W$ is a set of such pairs $(a,q)$ 
    with the property that for any two pairs $(a,q),(a',q') \in W$, the sets $M_{a,q}$ and $M_{a',q'}$ are disjoint.
    If, in addition, the sets $M_{a,q}$ partition the positive integers, \begin{equation} \mathbb{N} = \bigsqcup_{(a,q) \in W} 
\label{eq:disjunion} M_{a,q},\end{equation} then we say that the pairs $W$ form a \textit{smooth-rough-divisor-partition} of the integers.  For such a set $W$, it will sometimes be convenient to consider separately the two subsets $W_<\coloneq \{(a,q) \in W : P^+(a)<q\}$ and $W_= \coloneqq \{(a,q) \in W : P^+(a)=q\}$, so $W=W_< \sqcup W_=$.
\end{definition}
For example, take $W=\{(1,7),(2,5),(3,5),(4,3),(5,5),(6,5),(8,2),(9,3),(18,3)\}$.    In this case we have \begin{align*}
    W_< &= \{(1,7),(2,5),(3,5),(4,3),(6,5)\}, \\ W_= &= \{(5,5),(8,2),(9,3),(18,3)\}
\end{align*} and it is straightforward to check that $W$ forms a smooth-rough-divisor-partition. 

%An easy way to check that a set $W$ satisfies \eqref{eq:disjunion} is to verify the $M_{a,q}$ are disjoint and then check that \[\sum_{(a,q) \in W} \left(\frac 1a \prod_{p<q}\left(1-\tfrac 1q\right)\right) = 1.\] 
Note that if $W$ is finite, then $W_=$ is necessarily nonempty; however it is possible for $W_=$ to be empty if $W_<$ is infinite.\footnote{For example $W=W_<=\{(2^i,3),i\geq 0\}$.}

\medskip 

Now, suppose that $W$ is any finite smooth-rough-divisor-partition of the integers.  Then
\begin{align}
d(\mathcal{A}) & = \sum_{(a,q) \in W } d\big(\{m \in M_{a,q}:h(m)\geq 2\}\big) \notag \\  
&= \sum_{(a,q) \in W_< } d\big(\{m \in M_{a,q}:h(m)\geq 2\}\big) + \sum_{(a,q) \in W_= } d\big(\{m \in M_{a,q}:h(m)\geq 2\}\big) \notag \\
&=\sum_{(a,q) \in W_< }\frac{1}{a} A_q\left(\frac{2}{h(a)}\right) + \sum_{(a,q) \in W_= } \frac{1}{a} d\big(\{m': P^-(m')\geq q,\ h(m'a)\geq 2\}\big). \label{eq:interimprovedAbound}
\end{align}
In the second sum the largest prime divisor of $a$ is $q$ 
(since $(a,q) \in W_=$) so we can't necessarily split $h(m'a)$ using the multiplicativity of $h$, as $q$ may divide $m'$ as well as $a$. 

To handle $W_=$ we break into two cases based on whether $h(a)\geq 2$.  First we treat those $(a,q) \in W_=$ with $h(a)\geq 2$.  In this case every element of $M_{a,q}$ is abundant, so we treat these elements exactly as \[\sum_{\substack{(a,q) \in W_= \\h(a)\geq 2 }} \frac{1}{a} d\big(\{m': P^-(m')\geq q,\ h(m'a)\geq 2\}\big) = \sum_{\substack{(a,q) \in W_= \\h(a)\geq 2 }} \frac{1}{a} \prod_{p<q}\left(1-\frac 1p\right).\]  Note that the sum above also serves as a lower bound for $d(\mathcal{A})$.  We now treat the remaining elements of $W_=$. Let $a_q$ and $m''$ denote the largest divisors of $a$ and $m'$ respectively that are not divisible by $q$. Partitioning the integers $m'=q^im''$ according to the power of the prime $q$ dividing $m'$, we can bound the density appearing in the second sum by 

\begin{align}
 \frac{1}{a} d\big(\{m'&: P^-(m') \geq q,\ h(m'a)\geq 2\}\big) \label{eq:W_2densityterm} \\& = \frac 1a \sum_{i=0}^\infty  d\left(\{q^i m'': P^-(m'') > q,\ h(q^im'' a) \geq 2\}\right)\notag \\
 &= \frac 1a\sum_{i=0}^\infty \frac{1}{q^i} d\left(\left\{m'': P^-(m'') > q,\ h(m'')\geq \frac{2}{h(q^i a)}\right\}\right) \notag \\
 &\leq \frac 1{a(1-1/q)} \cdot  A_{q+1}\left( \frac{2(1-1/q)}{h(a_q)}\right). \notag
\end{align}   

In the final step above, since $a$ is not coprime to $q$ (but $a_q$ is), we bounded
\[h(q^ia) \leq h(a_q)\left(1+\frac 1q + \frac 1{q^2} + \cdots \right) = \frac{h(a_q)}{(1-1/q)}.\]  Inserting this in \eqref{eq:interimprovedAbound} gives the upper bound \begin{align}
    d(\mathcal{A})  &\leq  \sum_{(a,q) \in W_< }\frac{1}{a} A_q\left(\frac{2}{h(a)}\right) + \sum_{\substack{(a,q) \in W_= \\h(a)\geq 2 }} \frac{1}{a} \prod_{p<q}\left(1-\frac 1p\right) + \sum_{\substack{(a,q) \in W_= \\h(a) < 2 }} \frac {A_{q+1}\left( \frac{2(1-1/q)}{h(a_q)}\right)}{a\,(1-1/q)} \label{eq:afinalupbd} 
\end{align}
while recalling the observation about the lower bound for $d(\mathcal{A})$ gives
\begin{equation}
    \sum_{\substack{(a,q) \in W_= \\h(a)\geq 2 }} \frac{1}{a} \prod_{p<q}\left(1-\frac 1p\right) \ \leq \ d(\mathcal{A}). \label{eq:afinallowbd} 
\end{equation}
\subsubsection{Bounding covering numbers}
A similar bound holds for $d(\mathcal{C})$.  In this case we can derive an equation identical to \eqref{eq:afinalupbd} with $c$ in place of $h$ in each of the sums by using Theorem \ref{thm:cleqch} in place of the multiplicativity of $h$ in the first sum.  In the first two sums we then bound $c$ by $c'$. In the third sum, however, we work a bit harder to bound the expression $c(m'a)$ appearing in the expression analogous to \eqref{eq:W_2densityterm}.  Writing $m'a=m''a_qq^j$ as above, we have $c(m'a) \leq h(m'')c(a_{q}q^j)$.  We then bound $c(a_{q}q^j)$ by a function $ \overline{c}(a_q,q)$ defined as follows.  

First, if $a_q$ is a Sun-almost-covering number and $\tau(a_q) \geq q-1$ (note that in this case the number $a_qq$ is already either a covering number or a Sun-almost-covering number) then we use the trivial bound $c(a_qq^j) \leq \overline{c}(a_q,q) \coloneqq 2$.  Otherwise we take $\ell$ to be the Sun-almost-covering divisor of $a_q$, defined as in Definition \ref{def:cprime} (which may still be $a_q$ itself
), and set $b=\frac{a_q}{\ell}$ and then consider 
\begin{align}
    c(a_{q}q^j) \leq c'(a_{q}q^j) &= 1+\frac{\ell-1}{\ell} + \frac{1}{\ell} \sum_{d\mid bq^j, d>1} \frac{B(\tau(\ell),\omega(d))}{d} \notag \\
    &= 1+\frac{\ell-1}{\ell} + \frac{1}{\ell} \sum_{d\mid b,\ d>1} \frac{B(\tau(\ell),\omega(d))}{d} + \frac{1}{\ell}\sum_{d\mid bq^j,\ q|d} \frac{B(\tau(\ell),\omega(d))}{d} \notag \\
    &= c'(a_q) + \frac{1}{\ell} \sum_{\substack{d\mid b} }\sum_{i=1}^j\frac{B(\tau(\ell),\omega(dq^i))}{dq^i} \notag \\
    &= c'(a_q) + \frac{1}{\ell} \left(\sum_{i=1}^jq^{-i}\!\right)\sum_{\substack{d\mid b} }\frac{B(\tau(\ell),\omega(d)+1)}{d} \notag \\ 
    &< c'(a_q) + \frac{1}{\ell} \left(\frac{1}{q-1}\right)\sum_{\substack{d\mid b} }\frac{B(\tau(\ell),\omega(d)+1)}{d} \ \eqqcolon \ \overline{c}(a_q,q). \label{eq:overlinec}
\end{align}
Now, having defined $\overline{c}(a_q,q)$ in both of the cases when $a_q$ is a Sun-almost-covering number and $\tau(a_q) \geq q-1$ (here $\overline{c}(a_q,q)=2$) and otherwise by \eqref{eq:overlinec}, we can bound $d(\mathcal{C})$ analogously to \eqref{eq:afinalupbd} as
\begin{equation}
    d(\mathcal{C}) \leq  \sum_{(a,q) \in W_< }\frac{1}{a} A_q\left(\frac{2}{c'(a)}\right) + \sum_{\substack{(a,q) \in W_= \\c'(a)\geq 2 }} \frac{1}{a} \prod_{p<q}\left(1-\frac 1p\right) + \sum_{\substack{(a,q) \in W_= \\c'(a) < 2 }} \frac {A_{q+1}\left( \frac{2}{\overline{c}(a_q,q)}\right)}{a(1-1/q)}  . \label{eq:cfinalupbd}
\end{equation}

\section{Upper bound computations} \label{sec:upbdcomp}

The bounds described in Section \ref{sec:flexBehrend} apply for any set $W$ satisfying \eqref{eq:disjunion}.  In fact, the approach of Del\'eglise is essentially to take, for fixed $q$,  the set $W$ to be the infinite set $W=\{(a,q):P^+(a)<q\}$ (note in this case the subset $W_=$ is empty). He then applies a different argument to handle the large values of $a$.  

Thus, to bound these densities, it suffices to describe a set $W$ and then calculate the corresponding expressions in \eqref{eq:afinalupbd} and \eqref{eq:cfinalupbd}, so the goal now is simply to pick sets $W$ that result in as low of a bound as possible in a minimal amount of time.  After some experimentation, we found that a computationally efficient way to select the sets $W$ so that all of the summands appearing in the $W_<$ sums of \eqref{eq:afinalupbd} and \eqref{eq:cfinalupbd} were of either roughly uniform size or correspond to values of a with $h(a)\geq 2$ (respectively $c'(a)\geq 2$) that were encountered along the way.  In the case of $\mathcal{A}$, these $a$ values with $h(a)\geq 2$ can also be used to compute a corresponding lower bound for $d(\mathcal{A})$; however, since $c'(a)$ is only an upper bound for $c(a)$, we don't similarly obtain a lower bound for $d(\mathcal{C})$ in the same manner. 

In these calculations we need an upper bound for the quantities $A_y(x)$, which we numerically compute following the methods of Del\'eglise as outlined in Section \ref{sec:behrenddeleglise}.  Unlike Del\'eglise, we will need to compute these values for a wide range of values of $y$.  We largely follow his methods, except that we use a much larger values of $y$ (and a wider range of moments) which requires strengthening some of his inequalities to handle larger primes.  The details are described in Appendix \ref{app:deleglise}.  In what follows here we denote by $\widetilde{A}_y(x) \geq A_y(x)$ the numerical upper bound obtained by these calculations for a fixed $x$ and $y$.

Our algorithm then proceeds as follows.  We first fix parameters $Q$ and $Z$ with $Q$ a prime. We will consider values of $q<Q$, and the parameter $Z$ will correspond roughly with the size of the summands we consider in the $W_<$ sums.  

When computing the density of abundant numbers, we define \[f(n,k) \coloneqq \frac{\widetilde{A}_{{q}_{{}_{k}}}\!\left(\scalebox{0.95}{$\tfrac{2}{h(n)}$}\right)}{nk}\] where $q_k$ denotes the $k$-th prime number. After some numerical experimentation, we found that this function was a good measure of how deep to traverse the tree of smooth numbers to obtain good numerical values. We start by taking the set of integers \[S= \left\{n:h(n)<2, q_k= P^+\hspace{-0.3mm}(n)<Q, f(n,k+1)>Z\right\} .\]

We then take \begin{align*} W_1 \, &=\,  \{(nq_k,q_k):n\in S,  P^+(n)\leq q_k < Q, \text{ and } h(nq_k)\geq 2\}\\[1pt]
W_2 \, &=\, \big\{(n,q_{k}):n\in S,\ q_k=\min\!\big(\{Q\} \cup \{q_k:q_k>P^+(n),\ h(nq_k)<2, \text{ and }f(n,k)\leq Z\}\big)\big\}\\[1pt]
W_3 \, &=\, \left\{(nq_k,q_{k}):(n,q_j) \in W_2,\ P^+(n) \leq q_k < q_j,\text{ and }f(nq_k, k+1) \leq Z\right\} \\[2pt]
W_= &=\, W_1 \cup W_3, \ \  W_< = W_2.
\end{align*}

We now check that, in this case, $W=W_< \cup W_=$ forms a smooth-rough-divisor-partition by showing that every integer $n$ is contained in $M_{a,q}$ for precisely one $(a,q)\in W$.  Fix an integer $n$ and factor $n$ as $n=p_1p_2\cdots p_{\Omega(n)}$ with $p_1\leq p_2 \leq \cdots \leq p_{\Omega(n)}$.  
Let $d$ be the divisor of $n$ obtained by multiplying together the smallest prime factors of $n$ until the last time we obtain a divisor of $n$ that is contained in $S$, \[ d=\max_{0\leq j \leq \Omega(n)}\left\{d':d' \prod_{i=1}^j p_i, d' \in S \right\}. \]
Now let $r=P^-(n/d)\geq P^+(d)$ be the smallest prime factor of $n$ omitted from $d$.  We proceed with several cases.  

First, if $r=\infty$ (i.e., $n=d$) or $r\geq Q$, then $n \in M_{d,q}$ for the unique $(d,q)\in W_2$ whose first element is $d$.  It is also clear that such an $n$ is not contained in $M_{a,q}$ for any $(a,q)$ in either $W_1$ or $W_3$. 

Otherwise, if $h(dr)\geq 2$, then $(dr,r) \in W_1$ and $n \in M_{dr,r}$, while the construction of $W_2$ and $W_3$ precludes the possibility that $n \in M_{a,q}$ for any $(a,q)$ in $W_2$ or $W_3$.

This leaves the case that $h(dr)<2$. Suppose $r=q_k$ is the $k$th prime number. The fact that $dr=dq_k \notin S$ means that it must be the case that $f(dq_k,k+1)\leq Z$.  Suppose first that $1< r=q_k = P^+(d)$. (So we have $r^2\mid n$.) By the construction of $W_2$, we have $(d,q) \in W_2$ for some $q>r$, and so we find that $(dr,r)\in W_3$ and $n\in M_{dr,r}$ while $n \not\in M_{a,q}$ for any other $(a,q)\in W$.

Finally, that brings us to the case that $P^+(d)<r<Q$.    Since $d \in S$, there is a unique $(d,q) \in W_2$ whose first element is $d$.  If $r\geq q$ then $n \in M_{d,q}$, otherwise if $r<q$ then, as before, since $dr \not\in S$, and $h(dr)< 2$, we have $f(dr,r)\leq Z$, so by the construction of $W_3$ we have $(dr,r) \in W_3$. 

\medskip

By taking the parameters $Z=2^{-78}$ and $Q=q_{1000000}=15\,485\,863$ we create a smooth-rough-divisor-partition $W$, which we use in \eqref{eq:afinallowbd} and \eqref{eq:afinalupbd} to give both the lower and upper bounds quoted in Theorem \ref{thm:abundantbds} after about 20000 hours of computation.\footnote{While computing this level of precision required a large distributed computation, the method is able to improve on the existing bounds very quickly.  Taking $Z=2^{-46}$ and $Q=q_{10000}=104729$ we recover the previously known lower bound \eqref{eq:kobayashibds} and improve the upper bound to $d(\mathcal{A})<0.247628$ in under 10 seconds, if tables of moments for $\widetilde{A}_y(x)$ have been precomputed. Taking $Z=2^{-64}$, and $Q= q_{100000} = 1299709$ gives us 6 digits of $d(\mathcal{A})$, namely $0.24761951	<d(\mathcal{A}) < 0.24761998$ in just 11 CPU hours of computation time.}  

We use a nearly identical construction to obtain a smooth-rough-divisor-partition $W$ to compute an upper bound for $d(\mathcal{C})$, replacing only each occurrence of $h(n)$ and $f(n,k)$ in the definitions of $S, W_1, W_2$, and $W_3$ by $c'(n)$ and $f'(n,k)  \coloneqq \frac{1}{nk}\widetilde{A}_{{p}_{{}_{k}}}\!\left(\scalebox{0.95}{$\tfrac{2}{c'(n)}$}\right)$ respectively.  Then taking $Z=2^{-60}$ and $Q=q_{50000}=224737$ in  \eqref{eq:afinalupbd} gives the upper bound of Theorem \ref{thm:covdensity}.

\section{Lower Bounds for \texorpdfstring{$d(\mathcal{C})$}{d(C)}} \label{sec:Lowbd}

While the method described above can be used to obtain lower bounds on the density of abundant numbers using \eqref{eq:afinallowbd}, the analogous method for covering numbers does not give a lower bound for $d(\mathcal{C})$.  The analogous sum (the middle of the three sums in \eqref{eq:cfinalupbd}) gives a lower bound for the density of those integers $n$ with $c'(n)\geq 2$, of which $\mathcal{C}$ is a subset.\footnote{The same parameters used to compute the upper bound for $d(\mathcal{C})$ give the lower bound 0.103369 for the density of numbers with $c'(n)\geq 2$. This may already exceed the true value of $d(\mathcal{C})$, so another interpretation of this value is as a lower bound for how small the upper bound for $d(\mathcal{C})$ can be made using this method.}  It may be the case that this density is the same as that of $d(\mathcal{C})$, though this seems unlikely.

Instead, since the sequence of primitive covering numbers seems to be relatively sparse (at least at first), we compute the lower bound by computing exactly the density of multiples of known small primitive covering numbers.  

Determining whether a given integer $n$ is a covering number seems to be computationally intractable in general.  It seems likely this problem may be NP-complete, since, as noted in \cite[Problem AN2]{Garey} and proven in \cite{Stockmeyer}, even the ``simultaneous incongruences'' problem of determining whether a fixed set $\{(a_1,m_1),\ldots (a_k,m_k)\}$ of moduli $m_i$ and residues $a_i \pmod{m_i}$, forms a covering system is already NP-complete.  The same remark applies for the problem of computing the function $c(n)$.

We determine the complete list of small primitive covering numbers by the following method.  We treat integers $n$ one at a time in increasing order using the following steps.
\begin{enumerate}[itemsep=4pt, parsep=0pt]
    \item We check whether $n$ is divisible by any of the known, smaller primitive covering numbers.  If so, $n$ is not primitive covering and we proceed to the next $n$.  
    \item If $h(n)\leq 2$ then $n$ is not a covering number and we proceed to the next $n$. 
    \item We determine 
the Sun-almost-covering divisor $\ell(n)$ which we use to compute $c'(n)$.  Likewise, if $c'(n)<2$ then $n$ is not a covering number and we proceed to the next $n$.
\item If $\frac{n}{\ell(n)}$ is prime, we check whether $n$ can be verified to be a covering number using Corollary \ref{cor:covfromalmostcov}.  If so, we add it to our list of primitive covering numbers and proceed to the next $n$. (We know $n$ is a primitive covering number from the check in step 1.)
\item Finally, if the previous steps have not resolved whether $n$ is a covering number, we construct an integer (binary) program of boolean variables with constraints that are satisfiable if and only if $n$ is a covering number. We then ask the \cite{gurobi} MIP solver to solve our model.  The solver returns either a solution, corresponding to a covering system (and hence $n$ is a primitive covering number), or it returns ``infeasible''\!, in which case $n$ is not a primitive covering number.       
\end{enumerate}

We discuss further the logistics of step (5) in Appendix \ref{app:ip}. Step (5) is computationally intensive, requiring in some cases several hours for a single value of $n$.   Fortunately, we can resolve most small values of $n$ using steps (1)--(4). The smallest $n$ which requires step (5) is $n=7700$ (which is not a covering number).

In this way we are able to identify all of the primitive covering numbers less than $n=773500 = 2^2 {\times} 5^3 {\times} 7 {\times} 13 {\times} 17$, the smallest number for which our MIP solver has not (yet) been able to give an answer.  With the exception of 773500, we determine all primitive covering numbers less than 1000000, which are listed in Table \ref{tab:primcovnums}.  By computing exactly the density of the multiples of the numbers in this table (except 773500), we obtain the lower bound 
$0.10323<d(\mathcal{C})$.

\section*{Acknowledgements}
The authors thank Carl Pomerance, Ofir Gorodetsky, Jonah Klein and Mits Kobayashi for helpful comments, the anonymous referees for carefully reading both the paper and the associated code and providing useful suggestions, and Mike O'Leary for access to computational resources.

\appendix
\section{Stronger versions of Del\'eglise's moment bounds} \label{app:deleglise}

As mentioned in Section \ref{sec:upbdcomp}, we numerically bound the quantities $A_y(x)$ by computed values $\widetilde{A}_y(x) \geq A_y(x)$ which are obtained by calculations which closely follow those used by Del\'eglise.  We extend his calculations in order to allow for higher moments $r$ and larger values of $y$.  In particular, starting with \eqref{eq:moment}, we bound the moment \[\mu_{y,r} = \lim\limits_{x \to \infty} \frac 1x \sum\limits_{n\leq x} h_y(n)^r = \prod\limits_{p\geq y}\left(1\, {+}\sum\limits_{i=1}^\infty\frac{\rho_{p,r,i}}{p^i}\right) \] where we have written \[\rho_{p,r,i} \coloneqq \left(1+\frac 1p{+} \cdots {+} \frac 1{p^{i}}\right)^r\!-\left(1+\frac 1p{+} \cdots {+}\frac 1{p^{i-1}}\right)^r.\]

The following are Lemmas 6.3 and 6.4 respectively of \cite{deleglise}.

\begin{lemma}[Del\'eglise] \label{lem:dlem6.3}
    For every integer $r$ and every prime $p$,
    \[\sum_{i > 0} \frac{\rho_{p,r,i}}{p^i} \leq  \frac{(1+1/p)^r-1}{p}+r\left(\frac{1}{1-1/p}\right)^{r-1}\frac{1}{p^4-p^2}.\]
\end{lemma}
\begin{lemma}[Del\'eglise] \label{lem:dlem6.4}
For every integer $r$ and every prime $p > 
\max(2r, 15)$ we have
\[\sum_{i>0} \frac{\rho_{p,r,i}}{p^i} \leq 1.31\frac{r}{p^2}.\]
\end{lemma}

\begin{lemma}
For every integer $r$ with $1\leq r < 5\times 10^7$
we have 
\[\prod_{p>10^8}\left(1+  \sum_{i>0} \frac{\rho_{p,r,i}}{p^i}\right) \leq \exp\left(\tfrac{7r}{10^{10}}\right).\]
\end{lemma}
\begin{proof}
    We follow closely the proof of Lemma 6.5 of \cite{deleglise}.  Set \[u=\prod_{p>10^8}\left(1+  \sum_{i>0} \frac{\rho_{p,r,i}}{p^i}\right).\]
    Now, $\ln u \leq 1.31r\sum_{p>10^8}\frac 1{p^2}$ by Lemma \ref{lem:dlem6.4}.  Thus, using computed values of $\sum_p \frac{1}{p^2}=0.45224742004106\ldots$, we compute that \[\sum_{p>10^8}\frac{1}{p^2}<5.1616\times 10^{-10},\]
    and so \[u<\exp\left(1.31\times 5.1616\times 10^{-10}\times r\right) < \exp\left(\tfrac{7r}{10^{10}}\right). \qedhere\]
\end{proof}
We can now use this and Lemma \ref{lem:dlem6.3} to bound 
\begin{align*}
    \mu_{y,r} = \prod\limits_{p\geq y}\left(1\, {+}\sum\limits_{i=1}^\infty\frac{\rho_{p,r,i}}{p^i}\right) < \exp\left(\tfrac{7r}{10^{10}}\right)\!\prod_{y\leq p<10^8} \raisebox{0.46ex}{$\displaystyle\left(\raisebox{-0.4ex}{$\displaystyle \! 1+ \frac{(1+1/p)^r-1}{p}+\frac{r\left(\frac{1}{1-1/p}\right)^{r-1}}{p^4-p^2}$}\right)$} \eqqcolon \widetilde{\mu}_{y,r}.
\end{align*}
For our computation, we precompute the values of $\widetilde{\mu}_{y,r}$ for all $y<Q$, the parameter defined in Section \ref{sec:upbdcomp} (which is always taken less than $10^8$), and each $r=2^i$ for $0\leq i \leq 25$.  

We then use the values $\widetilde{\mu}_{y,r}$ to compute the numerical bounds $\widetilde{A}_y(x)> A_y(x)$, starting from the bound \eqref{eq:Aymomentbd} for $A_y(x)$ as \begin{align*}
     \widetilde{A}_y(x) \coloneqq  \prod_{p< y}\left(1-\frac 1p\right) \min_{0\leq i\leq 25}\left\{ \frac{\widetilde{\mu}_{y,2^i} -1}{x^{2^i}-1}\right\}.
\end{align*}
The computations of the values of $\widetilde{\mu}_{y,r}$ and $\widetilde{A}_y(x)$ are performed with floating point arithmetic in C++ using the MPFR library with 80 bits of precision and consistent rounding to ensure that every operation is consistently rounded to produce an upper bound.

\section{Integer Programming} \label{app:ip}
In Section \ref{sec:Lowbd}, the final step of determining whether an integer $n$ is a covering number (if easier options have not been successful) is to set up an integer programming problem with constraints that are satisfiable if and only if $n$ is a covering number.  

We simplify the calculations substantially using Lemma \ref{lem:wlog-almost-cover}.  Let $\ell$ be the Sun-almost-covering divisor of $n$, with $n=\ell b$, and $\gcd(\ell,b)=1$.  Following the ideas used in the proof of Theorem \ref{thm:generalizedstirlingbd}, if the divisors of $n$ form a covering system we can assume, without loss of generality, that the divisors of $\ell$ cover every residue modulo $\ell$ except $0 \pmod{\ell}$.  So it remains to cover only the $m$ residues modulo $n$ in this congruence class modulo $\ell$ using divisors of $n$ not coprime to $b$.  By the same argument as that proof, we can now reduce the question of determining whether $n$ is a covering number to the problem of determining whether it is possible to find a covering system $C'$ in which each modulus is a divisor of $b$ (greater than 1) and each modulus occurs at most $\tau(\ell)$ times.  This is the question we now resolve using integer programming.

For each divisor $d|b$, $d>1$, we create $d$ integer (binary) variables,\begin{equation}
    0\leq x_{d,i} \leq 1  \ \text{ for }0\leq i <d \label{eq:binaryconstraint}
\end{equation}
which will take on the value 1 when the residue class $i \pmod{d}$ is covered using one of the divisors of $d$.  Since each divisor $d$ can appear at most $\tau(\ell)$ times we impose the constraint \begin{equation} \sum_{i=0}^{d-1}x_{d,i} \leq \tau(\ell) \text{\ \  for each }d|b, \ d>1.\label{eq:divisorconstraint}\end{equation}

We must also ensure that each residue modulo $b$ is covered by this system.  We check this by imposing one constraint for each residue $j$ modulo $b$, namely
\begin{equation}\sum_{\substack{d|b\\d>1}} x_{d,j\text{(mod $d$)}} \geq 1 \text{\ \  for each }0\leq j < b.\label{eq:residueconstraint}\end{equation}

After creating the model described by \eqref{eq:binaryconstraint}, \eqref{eq:divisorconstraint} and \eqref{eq:residueconstraint} with $\sigma(b)-1$ variables and with $\sigma(b)+\tau(b)+b-2$ constraints, we ask a MIP solver (in our case Gurobi) to solve the model.  If it returns a solution we can use that solution to construct a covering system using the divisors of $n$ and if it returns ``infeasible" then we know that $n$ is not a covering number.  

For example, if this method were to be used to test whether $n=60=2^2 \times 3 \times 5$ is a covering number, we have the Sun-almost-covering divisor $\ell = \ell(60)=4$, $b=15$. We would have 23 variables ($x_{3,0},x_{3,1},x_{3,2},x_{5_0},x_{5,1},\ldots, x_{15,14})$, a constraint $0\leq x_{i,j}\leq 1$ for each variable and the constraints

\begin{minipage}{0.5\textwidth}\begin{align*}
    x_{3,0} + x_{3,1} + x_{3,2} \leq 3 \\
    x_{5,0} + x_{5,1} + x_{5,2} + x_{5,3} + x_{5,4} \leq 3 \\
    x_{15,0} + x_{15,1} + x_{15,2} + \cdots + x_{15,14} \leq 3\\ 
    \end{align*}\end{minipage}
    \begin{minipage}{0.5\textwidth}\begin{align*}
    x_{3,0} + x_{5,0} + x_{15,0} &\geq 1\\
    x_{3,1} + x_{5,1} + x_{15,1} &\geq 1\\
    x_{3,2} + x_{5,2} + x_{15,2} &\geq 1\\
    x_{3,0} + x_{5,3} + x_{15,3} &\geq 1\\[-5pt] \vdots\ \ \ \ \ \ \ \ &\\[-3pt]
    x_{3,2} + x_{5,4} + x_{15,14} &\geq 1.\\
\end{align*}\end{minipage}
In this case there are many solutions, including both $x_{3,0}=x_{3,1}=x_{3,2}=1$ and all other variables zero, or $x_{3,1}=x_{3,2}=x_{5,0}=x_{5,1}=x_{15,3}=x_{15,9}=x_{15,12}=1$, all others zero.  Thus 60 is a covering number, however it isn't a primitive covering number as it is divisible by 12.

\renewcommand{\baselinestretch}{1} 
\section{Tables}
\footnotesize

\begin{table}[!htbp]
    \centering    
\begin{tabular}{c|c|c|c}
$x$ &Primitive covering numbers up & Primitive covering numbers up & $\mathcal{P}_{\mathcal{C}}(x)$\\

& to $x$ found using Theorem \ref{thm:sunPrimitive} &to $x$ found using Theorem \ref{thm:strongersun} &\\
\hline
  $10^2$   & 3 & 3&3\\
$10^3$   & 7 & 7&7\\ 
$10^4$   & 11 & 13&13\\
$10^5$   & 19 & 35&45\\
$10^6$   & 25 & 58 &94 \text{\,or\,} 95\\ 
$10^{10}$   & 56 & 362&\\
$10^{20}$   & 214 & 6544&\\
$10^{30}$   & 542 &67153&\\
$10^{40}$   & 949 & 473340&\\
$10^{50}$   & 1433& 2592765&\\
\end{tabular} \vspace{2mm}
\caption{Table of counts of primitive covering numbers and of the subsets of the primitive covering numbers that can be found using Theorems \ref{thm:sunPrimitive} and \ref{thm:strongersun}.}
    \label{tab:primcovcounts}
\end{table}

\footnotesize
\begin{table}[!htbp]
    \centering
    \renewcommand{\arraystretch}{0.975}
\begin{tabular}{c|c|c||c|c|c}
$n$ & Factors & Determination & $n$ & Factors & Determination \\
\hline
12 & $ 2^{2} \cdot 3 $ & Theorem \ref{thm:sunPrimitive} & 112112 & $ 2^{4} \cdot 7^{2} \cdot 11 \cdot 13 $ & MIP Solver \\
80 & $ 2^{4} \cdot 5 $ & Theorem \ref{thm:sunPrimitive} & 117306 & $ 2 \cdot 3^{2} \cdot 7^{3} \cdot 19 $ & Theorem \ref{thm:strongersun} \\
90 & $ 2 \cdot 3^{2} \cdot 5 $ & Theorem \ref{thm:sunPrimitive} & 120042 & $ 2 \cdot 3^{5} \cdot 13 \cdot 19 $ & Corollary \ref{cor:covfromalmostcov} \\
210 & $ 2 \cdot 3 \cdot 5 \cdot 7 $ & Theorem \ref{thm:sunPrimitive} & 131274 & $ 2 \cdot 3^{3} \cdot 11 \cdot 13 \cdot 17 $ & MIP Solver \\
280 & $ 2^{3} \cdot 5 \cdot 7 $ & Theorem \ref{thm:sunPrimitive} & 142002 & $ 2 \cdot 3^{2} \cdot 7^{3} \cdot 23 $ & Theorem \ref{thm:sunPrimitive} \\
378 & $ 2 \cdot 3^{3} \cdot 7 $ & Theorem \ref{thm:sunPrimitive} & 145314 & $ 2 \cdot 3^{5} \cdot 13 \cdot 23 $ & Theorem \ref{thm:strongersun} \\
448 & $ 2^{6} \cdot 7 $ & Theorem \ref{thm:sunPrimitive} & 192500 & $ 2^{2} \cdot 5^{4} \cdot 7 \cdot 11 $ & MIP Solver \\
1386 & $ 2 \cdot 3^{2} \cdot 7 \cdot 11 $ & Theorem \ref{thm:strongersun} & 208544 & $ 2^{5} \cdot 7^{3} \cdot 19 $ & Corollary \ref{cor:covfromalmostcov} \\
1650 & $ 2 \cdot 3 \cdot 5^{2} \cdot 11 $ & Theorem \ref{thm:sunPrimitive} & 223074 & $ 2 \cdot 3^{8} \cdot 17 $ & Theorem \ref{thm:sunPrimitive} \\
2200 & $ 2^{3} \cdot 5^{2} \cdot 11 $ & Theorem \ref{thm:sunPrimitive} & 242250 & $ 2 \cdot 3 \cdot 5^{3} \cdot 17 \cdot 19 $ & Corollary \ref{cor:covfromalmostcov} \\
2464 & $ 2^{5} \cdot 7 \cdot 11 $ & Theorem \ref{thm:strongersun} & 252448 & $ 2^{5} \cdot 7^{3} \cdot 23 $ & Theorem \ref{thm:sunPrimitive} \\
5346 & $ 2 \cdot 3^{5} \cdot 11 $ & Theorem \ref{thm:sunPrimitive} & 272272 & $ 2^{4} \cdot 7 \cdot 11 \cdot 13 \cdot 17 $ & MIP Solver \\
9750 & $ 2 \cdot 3 \cdot 5^{3} \cdot 13 $ & Theorem \ref{thm:sunPrimitive} & 293250 & $ 2 \cdot 3 \cdot 5^{3} \cdot 17 \cdot 23 $ & Corollary \ref{cor:covfromalmostcov} \\
11264 & $ 2^{10} \cdot 11 $ & Theorem \ref{thm:sunPrimitive} & 311168 & $ 2^{7} \cdot 11 \cdot 13 \cdot 17 $ & MIP Solver \\
11466 & $ 2 \cdot 3^{2} \cdot 7^{2} \cdot 13 $ & Theorem \ref{thm:strongersun} & 318500 & $ 2^{2} \cdot 5^{3} \cdot 7^{2} \cdot 13 $ & MIP Solver \\
13000 & $ 2^{3} \cdot 5^{3} \cdot 13 $ & Theorem \ref{thm:sunPrimitive} & 323000 & $ 2^{3} \cdot 5^{3} \cdot 17 \cdot 19 $ & Corollary \ref{cor:covfromalmostcov} \\
14994 & $ 2 \cdot 3^{2} \cdot 7^{2} \cdot 17 $ & Theorem \ref{thm:strongersun} & 369750 & $ 2 \cdot 3 \cdot 5^{3} \cdot 17 \cdot 29 $ & Theorem \ref{thm:strongersun} \\
18954 & $ 2 \cdot 3^{6} \cdot 13 $ & Theorem \ref{thm:sunPrimitive} & 385434 & $ 2 \cdot 3^{2} \cdot 7^{2} \cdot 19 \cdot 23 $ & Corollary \ref{cor:covfromalmostcov} \\
20384 & $ 2^{5} \cdot 7^{2} \cdot 13 $ & Corollary \ref{cor:covfromalmostcov} & 391000 & $ 2^{3} \cdot 5^{3} \cdot 17 \cdot 23 $ & Corollary \ref{cor:covfromalmostcov} \\
23166 & $ 2 \cdot 3^{4} \cdot 11 \cdot 13 $ & Corollary \ref{cor:covfromalmostcov} & 395250 & $ 2 \cdot 3 \cdot 5^{3} \cdot 17 \cdot 31 $ & Theorem \ref{thm:strongersun} \\
26656 & $ 2^{5} \cdot 7^{2} \cdot 17 $ & Theorem \ref{thm:strongersun} & 423500 & $ 2^{2} \cdot 5^{3} \cdot 7 \cdot 11^{2} $ & MIP Solver \\
27846 & $ 2 \cdot 3^{2} \cdot 7 \cdot 13 \cdot 17 $ & Theorem \ref{thm:strongersun} & 431250 & $ 2 \cdot 3 \cdot 5^{5} \cdot 23 $ & Theorem \ref{thm:sunPrimitive} \\
30294 & $ 2 \cdot 3^{4} \cdot 11 \cdot 17 $ & Theorem \ref{thm:strongersun} & 450846 & $ 2 \cdot 3^{4} \cdot 11^{2} \cdot 23 $ & Corollary \ref{cor:covfromalmostcov} \\
31122 & $ 2 \cdot 3^{2} \cdot 7 \cdot 13 \cdot 19 $ & Theorem \ref{thm:strongersun} & 452608 & $ 2^{11} \cdot 13 \cdot 17 $ & Corollary \ref{cor:covfromalmostcov} \\
33150 & $ 2 \cdot 3 \cdot 5^{2} \cdot 13 \cdot 17 $ & Theorem \ref{thm:strongersun} & 485982 & $ 2 \cdot 3^{2} \cdot 7^{2} \cdot 19 \cdot 29 $ & Theorem \ref{thm:strongersun} \\
33858 & $ 2 \cdot 3^{4} \cdot 11 \cdot 19 $ & Theorem \ref{thm:strongersun} & 493000 & $ 2^{3} \cdot 5^{3} \cdot 17 \cdot 29 $ & Theorem \ref{thm:strongersun} \\
36608 & $ 2^{8} \cdot 11 \cdot 13 $ & Jones and White \cite{Jones17} & 505856 & $ 2^{11} \cdot 13 \cdot 19 $ & Corollary \ref{cor:covfromalmostcov} \\
37050 & $ 2 \cdot 3 \cdot 5^{2} \cdot 13 \cdot 19 $ & Theorem \ref{thm:strongersun} & 519498 & $ 2 \cdot 3^{2} \cdot 7^{2} \cdot 19 \cdot 31 $ & Theorem \ref{thm:strongersun} \\
37674 & $ 2 \cdot 3^{2} \cdot 7 \cdot 13 \cdot 23 $ & Theorem \ref{thm:strongersun} & 527000 & $ 2^{3} \cdot 5^{3} \cdot 17 \cdot 31 $ & Theorem \ref{thm:strongersun} \\
44200 & $ 2^{3} \cdot 5^{2} \cdot 13 \cdot 17 $ & Corollary \ref{cor:covfromalmostcov} & 568458 & $ 2 \cdot 3^{4} \cdot 11^{2} \cdot 29 $ & Theorem \ref{thm:strongersun} \\
44850 & $ 2 \cdot 3 \cdot 5^{2} \cdot 13 \cdot 23 $ & Theorem \ref{thm:strongersun} & 575000 & $ 2^{3} \cdot 5^{5} \cdot 23 $ & Theorem \ref{thm:sunPrimitive} \\
49400 & $ 2^{3} \cdot 5^{2} \cdot 13 \cdot 19 $ & Theorem \ref{thm:strongersun} & 612352 & $ 2^{11} \cdot 13 \cdot 23 $ & Theorem \ref{thm:strongersun} \\
49504 & $ 2^{5} \cdot 7 \cdot 13 \cdot 17 $ & Corollary \ref{cor:covfromalmostcov} & 617526 & $ 2 \cdot 3^{2} \cdot 7 \cdot 13^{2} \cdot 29 $ & Theorem \ref{thm:strongersun} \\
53248 & $ 2^{12} \cdot 13 $ & Theorem \ref{thm:sunPrimitive} & 654500 & $ 2^{2} \cdot 5^{3} \cdot 7 \cdot 11 \cdot 17 $ & MIP Solver \\
53900 & $ 2^{2} \cdot 5^{2} \cdot 7^{2} \cdot 11 $ & MIP Solver & 660114 & $ 2 \cdot 3^{2} \cdot 7 \cdot 13^{2} \cdot 31 $ & Theorem \ref{thm:strongersun} \\
55328 & $ 2^{5} \cdot 7 \cdot 13 \cdot 19 $ & Corollary \ref{cor:covfromalmostcov} & 685216 & $ 2^{5} \cdot 7^{2} \cdot 19 \cdot 23 $ & Corollary \ref{cor:covfromalmostcov} \\
59800 & $ 2^{3} \cdot 5^{2} \cdot 13 \cdot 23 $ & Theorem \ref{thm:strongersun} & 731500 & $ 2^{2} \cdot 5^{3} \cdot 7 \cdot 11 \cdot 19 $ & MIP Solver \\
63750 & $ 2 \cdot 3 \cdot 5^{4} \cdot 17 $ & Theorem \ref{thm:sunPrimitive} & 735150 & $ 2 \cdot 3 \cdot 5^{2} \cdot 13^{2} \cdot 29 $ & Theorem \ref{thm:strongersun} \\
66976 & $ 2^{5} \cdot 7 \cdot 13 \cdot 23 $ & Theorem \ref{thm:strongersun} & 747954 & $ 2 \cdot 3^{9} \cdot 19 $ & Theorem \ref{thm:sunPrimitive} \\
71250 & $ 2 \cdot 3 \cdot 5^{4} \cdot 19 $ & Theorem \ref{thm:sunPrimitive} & \textcolor{gray}{\textit{773500}} & \textcolor{gray}{$ \mathit{2^{2} \cdot 5^{3} \cdot 7 \cdot 13 \cdot 17}$} & \textcolor{gray}{\textit{Unknown Status}}\\
72930 & $ 2 \cdot 3 \cdot 5 \cdot 11 \cdot 13 \cdot 17 $ & MIP Solver & 785850 & $ 2 \cdot 3 \cdot 5^{2} \cdot 13^{2} \cdot 31 $ & Theorem \ref{thm:strongersun} \\
85000 & $ 2^{3} \cdot 5^{4} \cdot 17 $ & Theorem \ref{thm:sunPrimitive} & 863968 & $ 2^{5} \cdot 7^{2} \cdot 19 \cdot 29 $ & Corollary \ref{cor:covfromalmostcov} \\
95000 & $ 2^{3} \cdot 5^{4} \cdot 19 $ & Theorem \ref{thm:sunPrimitive} & 885500 & $ 2^{2} \cdot 5^{3} \cdot 7 \cdot 11 \cdot 23 $ & MIP Solver \\
95744 & $ 2^{9} \cdot 11 \cdot 17 $ & Corollary \ref{cor:covfromalmostcov} & 896610 & $ 2 \cdot 3 \cdot 5 \cdot 11^{2} \cdot 13 \cdot 19 $ & MIP Solver \\
97240 & $ 2^{3} \cdot 5 \cdot 11 \cdot 13 \cdot 17 $ & MIP Solver & 909568 & $ 2^{8} \cdot 11 \cdot 17 \cdot 19 $ & MIP Solver \\
100100 & $ 2^{2} \cdot 5^{2} \cdot 7 \cdot 11 \cdot 13 $ & MIP Solver & 923552 & $ 2^{5} \cdot 7^{2} \cdot 19 \cdot 31 $ & Theorem \ref{thm:strongersun} \\ 
107008 & $ 2^{9} \cdot 11 \cdot 19 $ & Theorem \ref{thm:strongersun} & 980200 & $ 2^{3} \cdot 5^{2} \cdot 13^{2} \cdot 29 $ & Theorem \ref{thm:strongersun} \\
107406 & $ 2 \cdot 3^{5} \cdot 13 \cdot 17 $ & Corollary \ref{cor:covfromalmostcov} \\
\end{tabular}\vspace{2mm}
    \caption{All primitive covering numbers (and candidates) up to one million, and how this was determined. (The primitivity of those numbers listed as Corollary \ref{cor:covfromalmostcov} or MIP Solver was checked by noting that they weren't divisible by anything earlier in the list.) Note the status of 773500 is not known.}
    \label{tab:primcovnums}\vspace{-3mm}
\end{table}

\pagebreak

\bibliographystyle{amsplain}
\bibliography{bibliography}
\vspace{1cm}

\end{document}